\documentclass[11pt]{amsart}
\usepackage{ulem}
\title{Determinacy and J\'onsson cardinals in $\lr$}

\author{S. Jackson}
\email{jackson@unt.edu}
\address{Steve Jackson\\
Department of Mathematics\\
University of North Texas\\
Denton, TX 76203}

\author{R. Ketchersid}
\address{Richard Ketchersid\\
Department of Mathematics\\
University of Arizona\\
Tucson, AZ 85721}

\author{F. Schlutzenberg}
\email{farmer.schlutzenberg@gmail.com}

\author{W. H. Woodin}
\email{woodin@math.harvard.edu}
\address{W. Hugh Woodin\\
Professor of Mathematics and of Philosophy\\
Harvard University\\
Cambridge MA 02138}

\usepackage{amssymb}
\usepackage{url}
\usepackage{lineno}
\usepackage[retainorgcmds]{IEEEtrantools}
\usepackage{color}

\newcommand{\ad}{\mathsf{AD}}
\newcommand{\AD}{\ad}
\newcommand{\lr}{L(\mathbb{R})}
\newcommand{\cof}{\mathrm{cof}}
\newcommand{\hod}{\mathrm{HOD}}
\newcommand{\HOD}{\hod}

\newcommand{\mws}{\sM_\omega^{\#}}

\newcommand{\zfc}{\mathsf{ZFC}}
\newcommand{\ZFC}{\zfc}
\newcommand{\J}{\mathcal{J}}

\newcommand{\ZF}{\mathsf{ZF}}

\newcommand{\onto}{\twoheadrightarrow}

\newcommand{\Hull}{\mathrm{Hull}}
\newcommand{\crit}{\mathrm{crit}}

\newcommand{\lh}{\mathrm{lh}}
\newcommand{\col}{\mathrm{Coll}}

\newcommand{\Coll}{\col}

\newcommand{\om}{\omega}
\newcommand{\sats}{\models}
\newcommand{\OR}{\mathrm{OR}}
\newcommand{\es}{\mathbb{E}}
\newcommand{\conc}{\ \widehat{\ }\ }
\newcommand{\inter}{\cap}
\newcommand{\un}{\cup}
\newcommand{\sub}{\subseteq}
\newcommand{\Ult}{\mathrm{Ult}}
\newcommand{\rg}{\mathrm{rg}}

\newcommand{\pred}{\mathrm{-pred}}
\newcommand{\cross}{\times}
\newcommand{\RR}{\mathbb{R}}

\newcommand{\com}{\circ}
\newcommand{\ran}{\text{ran}}
\newcommand{\eps}{\epsilon}
\newcommand{\elem}{\prec}

\newcommand{\all}{\forall}
\newcommand{\card}{\mathrm{card}}

\newcommand{\sM}{\mathcal{M}}
\newcommand{\sN}{\mathcal{N}}
\newcommand{\sT}{\mathcal{T}}

\newcommand{\sF}{\mathcal{F}}
\newcommand{\sQ}{\mathcal{Q}}

\newcommand{\pow}{\mathcal{P}}
\newcommand{\HC}{\mathrm{HC}}
\newcommand{\rest}{\upharpoonright}
\newcommand{\cut}{\backslash}
\newcommand{\kappabar}{\bar{\kappa}}
\newcommand{\id}{\mathrm{id}}

\newcommand{\ins}{\unlhd}

\swapnumbers
\newtheorem{thm}{Theorem}[section]
\newtheorem{lem}[thm]{Lemma}
\newtheorem{cor}[thm]{Corollary}
\swapnumbers
\newtheorem{case}{Case}
\newtheorem{subcase}{Subcase}

\theoremstyle{definition}
\newtheorem{claim}{Claim}
\newtheorem*{claim*}{Claim}
\swapnumbers
\newtheorem{dfn}[thm]{Definition}
\newtheorem{rem}[thm]{Remark}
\newtheorem{ques}[thm]{Question}

\newcommand{\bfSigma}{\mathbf{\Sigma}}

\begin{document}

\begin{abstract}
Assume $\ZF+\ad+V=\lr$ and let $\kappa<\Theta$ be an uncountable cardinal. We
show that $\kappa$ is J\'onsson, and
that if $\cof(\kappa)=\om$ then $\kappa$ is Rowbottom. We also establish some
other partition properties.
\end{abstract}
\maketitle
\section{Introduction} \label{sec_introduction}

Assume $\ad+V=\lr$ and let $\kappa<\Theta$ be an uncountable
cardinal. We will show that $\kappa$ is J\'onsson. We also show that if
$\cof(\kappa)=\om$ then $\kappa$ is Rowbottom. If $\kappa$ is regular then by
Steel's result, $\kappa$ is measurable and so Rowbottom (see
\cite[8.27]{outline}). So $\kappa$ is Rowbottom iff
$\cof(\kappa)\in\{\om,\kappa\}$.
However, we also show that irrespective of its cofinality,
$\kappa$ satisfies a partition property generalizing Rowbottomness, and also
satisfies another partition property, superficially stronger than J\'onssonness.

The history of these results are as follows. Kleinberg showed by pure
determinacy arguments that all (uncountable cardinals) $\kappa<\aleph_\om$ are
J\'onsson and that $\aleph_\om$ is Rowbottom (see \cite{kleinberg}). Jackson,
extending some work
joint with B. L\"owe, showed, also by determinacy arguments, that all
$\kappa<\aleph_{\omega_1}$ are J\'onsson, and
any $\kappa <\aleph_{\omega_1}$ of cofinality $\omega$ is Rowbottom.
Woodin then announced the result that all $\kappa<\Theta$ are
J\'onsson, and that this could be shown using the directed system
analysis of $\hod$. Later, motivated by Woodin's announcement, Jackson,
Ketchersid, and Schlutzenberg independently proved the same result in joint work,
along with
the Rowbottom and other partition results, also through the directed system
analysis. The proof depends on Lemma~\ref{strongerlem}, though 
only makes direct use of the weaker
Corollary~\ref{weakercor}. Our original argument established this
corollary directly. Schlutzenberg proved the stronger Lemma~\ref{strongerlem},
and we included this as it may be of independent interest.

Woodin has in fact proved that assuming $\ZF+\AD^+$, every uncountable cardinal
$\kappa<\Theta$ is J\'onsson. Here we limit ourselves to assuming $V=\lr$.

We now give some notation and recall some definitions. For any set $X$ and
$n<\om$, $[X]^n$ denotes the set of subsets of $X$ of cardinality $n$, and
$[X]^{<\om}$ the finite subsets. We let $\| A\|$ denote the cardinality of $A$.
We use the following partition terminology. Let
$\kappa,\gamma,\delta$ be cardinals. We
write
\[ [\kappa]^{<\om}_\delta\rightarrow[\kappa]^{<\om}_\gamma \]
iff for every function $F\colon[\kappa]^{<\om}\to\delta$ there is $A\sub\kappa$
such that
\[ \|A\|=\kappa\ \ \mathrm{and}\ \ \|F``[A]^{<\om}\|\leq\gamma.\]
We also write
\[ [\kappa]^{<\om}_{<\delta}\rightarrow[\kappa]^{<\om}_{\gamma} \]
iff $[\kappa]^{<\om}_\lambda\rightarrow[\kappa]^{<\om}_{\gamma}$ for each
cardinal $\lambda<\delta$. The notation
$[\kappa]^n_\delta\rightarrow[\kappa]^n_\gamma$, etc, is defined similarly.

Let $\kappa$
be an uncountable cardinal. Recall that:

$\kappa$ is \textit{Rowbottom} iff
$[\kappa]^{<\om}_{<\kappa}\rightarrow[\kappa]^{<\om}_\om$.

$\kappa$ is \textit{J\'onsson} iff for every $F \colon
[\kappa]^{<\omega} \to \kappa$
there is $A \subseteq \kappa$ such that
\[ \|A\|=\kappa\ \ \mathrm{and}\ \ F``[A]^{<\omega} \neq \kappa.\]

\section{Main Results}

We now state our main results.

\begin{thm}\label{mainthm}
Assume $\ad+V=\lr$. Let $\kappa<\Theta$ be an uncountable
cardinal. Then:
\begin{itemize}
\item[\textup{(}a\textup{)}] If $\cof(\kappa)=\omega$ then $\kappa$ is
Rowbottom.
\item[\textup{(}b\textup{)}] In fact, in general,
$[\kappa]^{<\om}_{<\kappa}\rightarrow[\kappa]^{<\omega}_{\cof(\kappa)}$ and
$[\kappa]^{<\om}_{<\cof(\kappa)}\rightarrow[\kappa]^{<\omega}_{\omega}$.
\item[\textup{(}c\textup{)}] $\kappa$ is J\'onsson.
\item[\textup{(}d\textup{)}] In fact, let $\lambda$ be a
cardinal such that $\om_1\leq\lambda\leq\kappa$. Let
\[ F\colon [\kappa]^{<\om}\to\lambda.\]
Then there is $A\sub\kappa$ such that:
\begin{itemize}
\item $\|A\|=\kappa$;
\item $\|\lambda\cut F``[A]^{<\om}\|=\lambda$; in fact,
$\lambda\cut F``[A]^{<\om}$ contains a club subset of $\lambda$ of cardinality
$\lambda$.
\end{itemize}
\end{itemize}
\end{thm}

As a corollary to the proof of Theorem \ref{mainthm}, we obtain a simultaneous
partition property:

\begin{thm}\label{simultaneous}
Assume $\ad+V=\lr$. Let $\kappa<\Theta$ be an uncountable cardinal,
$\gamma_1,\lambda_1<\kappa$ and $\gamma_2,\lambda_2<\cof(\kappa)$.
For each $i\in\{1,2\}$, let
$\left<F^i_\alpha\right>_{\alpha<\gamma_i}$ be such that for each
$\alpha<\gamma_i$ we have
\[ F^i_\alpha\colon [\kappa]^{<\om}\to\lambda_i. \]
Then there is $A\sub\kappa$ such that $\|A\|=\kappa$
and
\begin{IEEEeqnarray*}{l}
\all\alpha<\gamma_1\left(\|F^1_\alpha``[A]^{<\om}\|\leq\cof(\kappa)\right);\\
\all\alpha<\gamma_2\left(\|F^2_\alpha``[A]^{<\om}\|\leq\om\right).
\end{IEEEeqnarray*}\end{thm}

By a standard argument, Theorem~\ref{simultaneous} easily implies the following
``two-cardinal'' 
result.

\begin{cor}
Assume $\ad+V=\lr$. Let $\kappa<\Theta$ be a cardinal, and suppose
\[ \om\leq\lambda_1<\cof(\kappa)\leq\lambda_2<\kappa.\]
Then
$(\kappa,\lambda_2,\lambda_1)\rightarrow(\kappa,\cof(\kappa),\om)$. That
is, for any first order structure
\[ M=(\kappa,\vec R, \lambda_2,\lambda_1)\]
with
universe $\kappa$, countably many predicates $\vec R$, and one-place predicates
$\lambda_1,\lambda_2$, there is $X\elem M$, with $X$ having universe
$A\sub\kappa$ with $\|A\|=\kappa$, and
$\|A\inter\lambda_2\|\leq\cof(\kappa)$, and $\|A\inter\lambda_1\|\leq\om$.
\end{cor}

\begin{proof}
Let $G\colon [\kappa]^{<\om}\to\kappa$ be a Skolem function for $M$. (Take $G$
such that whenever $B\sub\kappa$ has limit ordertype, then
$G``[B]^{<\om}\elem M$.)  Let
\[ F_i\colon [\kappa]^{<\om}\to\lambda_i+1\]
be
defined by $F_i(b)=\min(G(b),\lambda_i)$. If $A\sub\kappa$ 
witnesses Theorem \ref{simultaneous} with respect to $F_1,F_2$ then 
$X=G``[A]^{<\omega}$ is as required.
\end{proof}

\begin{rem} The partition properties in Theorem \ref{mainthm} are optimal in
certain ways. The property
\[ [\kappa]^1_\kappa\rightarrow[\kappa]^1_\lambda\]
is false for any $\lambda<\kappa$ and
\[ [\kappa]^1_{\cof(\kappa)}\rightarrow[\kappa]^1_{\lambda}\]
is
false for
any $\lambda<\cof(\kappa)$.

In Theorem \ref{simultaneous}, if $\kappa$ is singular, the
requirement that the $F^1_\alpha$'s be uniformly bounded by
$\lambda_1$ is necessary. For let
$\left<\gamma_\alpha\right>_{\alpha<\cof(\kappa)}$ be $\kappa$-cofinal
and let $F_\alpha\colon \kappa\to\kappa$ be given by
$F_\alpha(\beta)=\beta$ for $\beta<\gamma_\alpha$, and
$F_\alpha(\beta)=0$ otherwise. There is no $A$ as in Theorem
\ref{simultaneous} for this sequence.
\end{rem}

Before we start the proofs, we mention a couple of related questions:

\begin{ques} Assume $\ad+V=L(\RR)$. Do the partition properties of \ref{mainthm}
hold for any $\kappa\geq\Theta$?
Are there nonordinal J\'{o}nsson cardinals? In particular, is $\RR$
J\'{o}nsson?
\end{ques}

Ralf Schindler suggested the following question.

\begin{ques} What is the consistency strength of $\ZF+$``Every uncountable
cardinal $\kappa$ (or $\kappa<\Theta$) is J\'{o}nsson''?\end{ques}

The proof of Theorems \ref{mainthm} and \ref{simultaneous} proceed
through a few lemmas.
We first use the directed system analysis of
$\hod^{\lr}$ to prove \ref{strongerlem}. Its proof is related
to the proof of Steel's result \cite[8.27]{outline}, that if $\kappa$ is regular
and uncountable,
then it is measurable (under the same hypotheses).

\begin{lem} \label{strongerlem}
Assume $\ad+V=\lr$. Let $\kappa<\Theta$ be a cardinal such
that $\om_1<\kappa$, and let $x\in\RR$.
Then $\hod_x\sats$``there are $\kappa$-many measurables $<\kappa$''.
\end{lem}

\begin{cor}\label{weakercor}
Adopt the assumptions of Lemma \ref{strongerlem} other than ``$\kappa>\om_1$'';
assume $\kappa>\om$.
Then $\hod_x\sats$``either $\kappa$ is measurable or is a limit of
measurables''.
\end{cor}

Working in $\zfc$, we then prove in Lemma \ref{secondlem}, some
partition properties for $\kappa$ such that $\kappa$ is either
measurable or a limit of measurables. We then prove the main theorems
by passing from $\lr$ into some $\hod_x^{\lr}$, applying Lemma
\ref{secondlem} there.

\section{Analysis of generators}
In this section we prove Lemma \ref{techlem}, which we need to prove
\ref{strongerlem}. This involves an analysis of generators produced by
certain iteration trees. We will deal in general with nonnormal, fine
iteration trees on premice, so as to give \ref{techlem} more generally.
However, for the purposes of proving \ref{strongerlem} it suffices to
consider only finite stacks of normal iteration trees.

We first discuss some nonstandard terminology and facts related to iteration
trees. For standard background, see \cite{fsit} and \cite{outline}.

Given a structure $\sN$ and $\mu\in\OR^{\sN}$, if $\mu$ is the largest
cardinal of $\sN$ then we let $(\mu^+)^{\sN}$ denote $\OR^{\sN}$.

Given a premouse $\sN$ and a limit ordinal $\alpha\leq\OR^\sN$, 
$\sN|\alpha$ denotes
the initial segment of $\sN$ of height $\alpha$, and
$\sN||\alpha$ its passive counterpart. Also, $F^\sN$ denotes the 
active 
extender of $\sN$, and $\es^{\sN}$ the 
extender sequence of $\sN$ ($\es^{\sN}$ does not include the active extender).

We take \textit{iteration tree} to be defined as
in \cite[Section 5]{fsit}, except that we drop condition (3)
(i.e., the condition ``$\alpha<\beta\implies\lh(E_\alpha)<\lh(E_\beta)$''), and
strengthen the first clause of condition (4), to:
\[ \text{``if}\ \sT\pred(\alpha+1)=\beta\ \text{then}\ 
\kappa=\crit(E_\alpha)<\rho_\gamma\ \text{for each
}\gamma\in[\beta,\alpha)\text{''}. \]
(Here $\rho_\gamma$ refers to $\nu(E_\gamma)$.) As in \cite[Section 5]{fsit},
$\sM^{*\sT}_{\alpha+1}$ denotes the ``preimage'' of 
$\sM^{\sT}_{\alpha+1}$, and is
some $\sN\ins \sM^{\sT}_{\sT\pred(\alpha+1)}$ such that 
\[ \pow(\kappa)\inter\sN=\pow(\kappa)\inter\sM^{\sT}_\alpha|\lh(E^{\sT}_\alpha).\] 
Any 
tree 
satisfying
all requirements of \cite[Section 5]{fsit} (including (3)), in fact satisfies the
strengthening of the first
  clause of (4) given above.

\begin{rem}\label{rem:*0} Let $\sT$ be an iteration tree and for $\alpha<\lh(\sT)$ let
$\sM_\alpha=\sM^\sT_\alpha$, $E_\alpha=E^\sT_\alpha$, etc. Let
$\delta+1<\lh(\sT)$ and $\beta=\sT\pred(\delta+1)$.  If $\sT$ is not
normal then $E_\delta$ might not be close to $\sM^*_{\delta+1}$,
but we still have preservation of the degree $n$ fine structure
between $\sM^*_{\delta+1}$ and $\sM_{\delta+1}$, where
$n=\deg(\delta+1)$. For instance,
$i^*_{\beta,\delta+1}(p_n^{\sM^*_{\delta+1}})=p_n^{\sM_{\delta+1}}$. See, for example, 
\cite[4.3,4.4]{fsit}. Also, letting $\mu=\crit(E_\delta)$,
\[ \sM_{\delta+1}|(\mu^+)^{\sM_{\delta+1}}=\sM^*_{\delta+1}||(\mu^+)^{\sM^*_{\delta+1}}. \]
\end{rem}

\begin{rem}\label{rem:*1} Let $\sT,\delta,\beta$ be as above and
$\mu=\crit(E_\delta)$.
Let $\gamma\in[\beta,\delta)$ be such that $\lh(E_\gamma)$ is
  minimal. Then we have (a), (b), (c) below:
  
   (a) $(\mu^+)^{\sM^*_{\delta+1}}\leq\lh(E_\eps)$ for
  each $\eps\in[\beta,\delta)$.
   
   (b) In fact:
   \begin{itemize}
    \item[-] For each
    $\eps\in[\beta,\gamma]$, $E_\gamma$ is on $\es^{\sM_\eps}\conc
    F^{\sM_\eps}$.
    \item[-] For each $\eps\in(\gamma,\delta]$,
  $\lh(E_\gamma)$ is a cardinal of $\sM_\eps$ and
  $\sM_\eps|\lh(E_\gamma)=\sM_\beta||\lh(E_\gamma)$.
  \item[-] If
  $\gamma'\in[\beta,\delta)$ and $\lh(E_{\gamma'})=\lh(E_\gamma)$ then
    $\gamma=\gamma'$.
    \end{itemize}
    
    Proof of (b): This follows by induction through
    $[\beta,\delta]$, using coherence, etc.
    
    Proof of (a): This follows from (b) since
    $E_\delta$ must measure exactly $\pow(\mu)\inter
    \sM^*_{\delta+1}$.
    
    (c) Suppose $(\mu^+)^{\sM^*_{\delta+1}}=\lh(E_\gamma)$. Then
    $\sM^*_{\delta+1}=\sM_\beta|\lh(E_\gamma)$, so $\mu$ is the 
largest
    cardinal of $\sM^*_{\delta+1}$ (recall that
    $\sM_\gamma|\lh(E_\gamma)$ projects strictly below $\lh(E_\gamma)$),
    $E_\gamma$ is type 2 (since $\mu<\nu(E_\gamma)$), and if also
    $\sT$ does not drop in model at $\delta+1$ then $\beta=\gamma$.\end{rem}

Given a premouse $\sM$ and $\kappa\in\OR^{\sM}$, say $\kappa$ is
\textit{finely measurable \textup{(}fm\textup{)} in $\sM$} iff there is
$E$ on $\es^{\sM}\conc F^{\sM}$ such that $\crit(E)=\kappa$ and $E$ is
total over $\sM$. Say $\kappa$ is \textit{almost finely
  measurable \textup{(}afm\textup{)} in $\sM$} if $\kappa$ is fm in
$\sM$ or if $\sM$ is active type $2$ and $\kappa$ is fm in
$\Ult_0(\sM,F^{\sM})$.

Let us make some observations on this definition.

\begin{rem}\label{rem:*2} Let $\sM$ be a premouse and $\kappa<\OR^\sM$. Then:

(a) If $\kappa$ is fm in $\sM$ then
$(\kappa^+)^\sM<\OR^\sM$.

(b) If $\sM$ is active and $\kappa$ is fm in
$U=\Ult_0(\sM,F^\sM)$ and $(\kappa^+)^\sM<\OR^\sM$ then $\kappa$ is fm
in $\sM$.

Proof of (a): If $\sN$ is active and
$\kappa=\crit(F^\sN)$ then $(\kappa^+)^\sN<\OR^{\sN}$.

Proof of (b): Let
$E$ witness the fine measurability of $\kappa$ in $U$, and let $G$ be
the normal measure segment of $E$. Then $G$ is on $\es^U\conc F^U$ and $(\kappa^+)^\sM<\OR^\sM$ 
and $\OR^\sM$ is a cardinal
of $U$ and $\sM||\OR^\sM=U|\OR^\sM$, which implies $\lh(G)<\OR^\sM$ and that $G$
is on $\es^\sM$, giving (b).\end{rem}

\begin{rem}\label{rem:*3} Let $\sT,\delta,\beta,\mu$ be as in \ref{rem:*1} and assume that
$\sT$ does not drop in model at $\delta+1$.  Then $\mu$ is afm in
$\sM_\beta$.

Proof: If $\beta=\delta$ the statement is trivial, so assume
$\beta<\delta$. Let $G$ be the normal measure segment of
$E_\delta$. We will show that $G$ witnesses the afm of $\mu$ in
$\sM_\beta$. We have $\sM^*_{\delta+1}=\sM_\beta$. Let $\gamma$ be as
in \ref{rem:*1}.  By \ref{rem:*1}(a), $(\mu^+)^{\sM_\beta}\leq\lh(E_\gamma)$.

Suppose first that $(\mu^+)^{\sM_\beta}$ exists in $\sM_\beta$. Using
\ref{rem:*1}(c),(b), we get that $(\mu^+)^{\sM_\beta}<\lh(E_\gamma)$ and
both are cardinals in $\sM_\delta$, so since $G$ is type 1 and $G$ is
$\sM_\beta$-total, $(\mu^+)^{\sM_\beta}<\lh(G)<\lh(E_\gamma)$, so $G$
is on $\es^{\sM_\beta}$, and $\mu$ is afm there.

Now suppose $\mu$ is the
largest cardinal of $\sM_\beta$. Then by (\ref{rem:*1})(a),(c), we have that
$E_\beta=F^{\sM_\beta}$ is type 2 and $\gamma=\beta$. Like in the previous case
(but considering the interval $[\beta+1,\delta)$)
then $G$ is on $\es^{\sM_{\beta+1}}$, and the agreement between
$\sM_{\beta+1}$ and $U=\Ult_0(\sM_\beta,E_\beta)$ gives that $G$ is on
$\es^U$.\end{rem}

\begin{rem}\label{rem:*4} Let $\pi:\sM\to\sN$ be
$\Sigma_1$-elementary between premice $\sM,\sN$, and let
$\tau\in\OR^\sM$. Suppose that if $\tau$ is the
largest cardinal of $\sM$ then $\pi(\tau)$ is likewise in $\sN$.  Then
$\tau$ is afm in $\sM$ iff $\pi(\tau)$ is afm in $\sN$.

Proof: First assume that $(\tau^+)^\sM$ exists in $\sM$. By
$\Sigma_1$-elementarity,
\[ \pi((\tau^+)^\sM)=(\pi(\tau)^+)^\sN.\]
So by
\ref{rem:*2} we need only consider fine measurability in $\sM,\sN$.  Now
$\tau$ is fm in $\sM$ iff there is $E$ on $\es^\sM\conc
F^\sM$ such that $\crit(E)=\tau$ and $(\tau^+)^\sM<\lh(E)$. This property is
$\Sigma_1$
in the parameter $(\tau^+)^\sM$, and it reflects to $\pi(\tau)$ in $\sN$.

Now suppose that $\tau$ is the largest cardinal of $\sM$,
so $\pi(\tau)$ is the largest cardinal of $\sN$.  So assume $\sM,\sN$
are type 2 and consider fine measurability in
$U^\sM=\Ult_0(\sM,F^\sM)$ and $U^\sN=\Ult_0(\sN,F^\sN)$. Let
\[ \psi\colon U^\sM\to U^\sN\]
be given by $\pi$ and the shift lemma. Then
$\psi$ is $\Sigma_1$-elementary and $\psi(\tau)=\pi(\tau)$, so the
statement reduces to the previous case. (If $U^\sM,U^\sN$ are not
wellfounded then the previous case doesn't literally apply, but the
first order properties of $U^\sM,U^\sN,\psi$ are sufficient.)\end{rem}

\begin{dfn} Let $\pi:\sM\to\sN$ be $\Sigma_1$-elementary between premice 
and $\gamma\in\OR^\sN$, such that $\gamma<\sup\pi``\OR^{\sM}$.

We say $\gamma$ is a \textit{generator} (relative to $\pi$) iff
$\gamma\neq\pi(f)(a)$ for any $f\in\sM$ and $a\in\gamma^{<\om}$.

We say that $\sN$ has the \textit{hull property at $\gamma$} (relative
to $\pi$) iff $\pow(\gamma)^\sN\sub H$, where $H$ is the transitive
collapse of $\Hull_0^\sN(\gamma\un\pi``\sM)$.
\end{dfn}

\begin{rem}\label{rem:*5} Let $\pi:\sM\to\sN$ be $\Sigma_1$-elementary and
$\gamma\in\OR^\sN$. Then we have (a), (b), (c), (d) below.

(a) If $\pi(\beta)>\gamma$, then $\gamma$ is a
generator iff $\gamma\neq\pi(f)(a)$ for all $f\colon \beta^{<\om}\to\beta$
and $a\in\gamma^{<\om}$.

(b) Suppose also $\sigma\colon \sN\to\sQ$ is
$\Sigma_1$-elementary. Then $\gamma$ is a generator for $\pi$ iff
$\sigma(\gamma)$ is a generator for $\sigma\com\pi$.

Proof of (b): Let $f\in\sM$,
and $\gamma\notin\pi(f)``\gamma^{<\om}$. This lifts under
$\sigma$, and vice versa.

(c) $\sN$ has the hull property
at $\gamma$ iff for every $A\in\pow(\gamma)^\sN$, there is $f\in\sM$
and $a\in\gamma^{<\om}$ such that $\pi(f)(a)\inter\gamma=A$.

(d) If $\sN$ is sufficiently iterable, has the hull property at $\gamma$, and $H$ is 
the transitive
collapse of $\Hull_0(\gamma \un\pi``\sM)$, then
$H||(\gamma^+)^H=\sN||(\gamma^+)^{\sN}$.

Proof of (d): Use condensation.

(However, it can happen that $H|(\gamma^+)^H\neq \sN|(\gamma^+)^{\sN}$.
For example, suppose $\sM$ is type 2, $\gamma$ is the largest cardinal of $\sM$
and $\gamma$ is afm in $\sM$.
Let $E\in\Ult_0(\sM,F^{\sM})$ witness the latter. Let $\sN=\Ult_0(\sM,E)$ and
let $\pi=i_E^{\sM}$. Then $\sN$ has the hull property at $\gamma$ and $H=\sM$,
so $H$ is active at $(\gamma^+)^H=(\gamma^+)^{\sN}$, but $\sN$ is passive
there.)\end{rem}

\begin{rem}\label{rem:normalstack}
The following lemma is the main result of this section. It applies to iteration
trees which aren't necessarily normal. However, for our intended application, one
may assume
that $\sT$ is a finite stack of normal trees $\sT_0,\ldots,\sT_{k-1}$,
and that $\eta,\xi$ are both indices of $\sT_{k-1}$.
\end{rem}

\begin{lem} \label{techlem}
Let $\sM$ be a fine-structural premouse, $\sT$ an iteration tree on
$\sM$, $\xi<\lh(\sT)$, $n=\deg^{\sT}(\xi)$.  Let
$\chi\in[0,\xi]_{\sT}$ be such that $(\chi,\xi]_\sT$ has no drops in
  model or degree. Let $\eta\in[\chi,\xi]_\sT$.

Let $\kappa\leq\rho_n^{\sM_\xi}$ be in the range of
$i^\sT_{\eta,\xi}$.  Let $\lambda$ be the sup of all $\gamma<\kappa$
such that $\gamma$ is afm in $\sM_\xi$.  It
follows that $\lambda\in\rg(i^\sT_{\eta,\xi})$ \textup{(}to be
established\textup{)}.

For $\alpha\in[\eta,\xi]_\sT$, let
$i^\sT_{\alpha,\xi}((\kappa_\alpha,\lambda_\alpha))=(\kappa,\lambda)$
and let $G_\alpha$ be the set of $i_{\chi,\alpha}$-generators in the
interval $[\lambda_\alpha,\kappa_\alpha]$.

\begin{enumerate}
\item[\textup{(}1\textup{)}] Suppose $\kappa<\rho_n^{\sM_\xi}$ and that 
$\sM_\eta$ has
the hull
property at every point in $[\lambda_\eta,\kappa_\eta]$, relative to
$i_{\chi,\eta}$. Then\textup{:}
\begin{enumerate}
 \item[\textup{(}a\textup{)}] $\sM_\xi$ has the hull property at every
point in $[\lambda_\xi,\kappa_\xi]$ relative to $i_{\chi,\xi}$.
 \item[\textup{(}b\textup{)}] For each $\alpha\in[\eta,\xi]_\sT$,
\[ i^\sT_{\alpha,\xi}``G_\alpha=G_\xi\inter\left(\{\kappa_\xi\}\un\sup
i_{\alpha,\xi}``\kappa_\alpha\right).\]
\item[\textup{(}c\textup{)}] If $\kappa$ is afm in $\sM_\xi$ then
$\{\kappa_\xi\}\un G_\xi\cut\sup i_{\eta,\xi}``\kappa_\eta$ is a
closed set of inaccessibles of $\sM_\xi$.
\item[\textup{(}d\textup{)}] If $\kappa$ is not afm in
$\sM_\xi$ then $G_\xi=i_{\eta,\xi}``G_\eta$.
\end{enumerate}
\item[\textup{(}2\textup{)}] Suppose $\kappa=\rho_n^{\sM_\xi}$. Then\textup{:}
\begin{enumerate}
\item[\textup{(}a\textup{)}] Either $G_\xi=\emptyset$ or
$G_\xi=\{\kappa_\xi\}$. Moreover, $G_\xi=i_{\eta,\xi}``G_\eta$.
\item[\textup{(}b\textup{)}]
\[ \sM_\xi|\rho_n^{\sM_\xi}\sub\Hull_0^{\sM_\xi}(\lambda_\xi\un
i_{\chi,\xi}``(\sM_\chi||\rho_n^{\sM_\chi})). \]
\end{enumerate}
\end{enumerate}
\end{lem}

\begin{proof} We start with the following claim, then prove parts (1) and (2).

\begin{claim} \label{c2} For $\alpha\in[\eta,\xi]_\sT$, let $\lambda_\alpha'$ be
the sup of all afm's $\gamma$ of $\sM_\alpha$ such that $\gamma<\kappa_\alpha$.
Then $i_{\alpha,\xi}(\lambda_\alpha')=\lambda=\lambda_\xi$. In particular,
$\lambda_\xi\in\rg(i_{\alpha,\xi})$, and $\lambda_\alpha=\lambda_\alpha'$.

Moreover, if $\lambda<\kappa=\rho_n^{\sM_\xi}$ then also
$\lambda\in\rg(i_{\chi,\xi})$.
\end{claim}

\begin{proof}
Note that if $\gamma$ is afm in $\sM$, then $\gamma$ is a limit
cardinal of $\sM$.

Let $\theta_\alpha$ be the largest limit cardinal $\theta$ of
$\sM_\alpha$ such that $\theta\leq\kappa_\alpha$.  So
$\theta_\xi=i_{\alpha,\xi}(\theta_\alpha)$ is likewise with respect to
$\sM_\xi$ and $\kappa_\xi$, and $\lambda_\alpha'\leq\theta_\alpha$ and
$\lambda_\xi'\leq\theta_\xi$.

By \ref{rem:*4}, $\theta_\alpha$ is afm in $\sM_\alpha$ iff $\theta_\xi$ is
afm is $\sM_\xi$. Moreover, let $\beta=\alpha$ or $\beta=\xi$. Given
$\gamma<\theta_\beta$, $\gamma$ is afm in $\sM_\beta$ iff $\gamma$ is
afm in $\sM_\beta|\theta_\beta$ (by the initial segment condition and
that $\theta_\beta$ is a limit cardinal of $\sM_\beta$). Therefore,
$\lambda_\beta'$ is definable over $\sM_\beta|\theta_\beta$ (possibly
$\lambda_\beta'=\theta_\beta$), uniformly in $\beta$.

It follows that
$i_{\alpha,\xi}(\lambda_\alpha')=\lambda_\xi'=\lambda_\xi$, as
required.

Finally, if $\kappa=\rho_n^{\sM_\xi}$ and $\lambda<\kappa$, then the fact that
$i_{\chi,\xi}``\rho_n^{\sM_\chi}$ is cofinal in $\rho_n^{\sM_\xi}$,
and the arguments above, show that $\lambda\in\rg(i_{\chi,\xi})$.
\end{proof}

We now prove (1). So assume $\kappa<\rho_n^{\sM_\xi}$ and the hull property
hypothesis of (1). Let (1)$_\xi$ be the conjunction of (1)(a)--(1)(d) (relative
to $\xi$). We proceed by induction on
$\alpha \in [\eta,\xi]_\sT$ to prove (1)$_\alpha$. The statement
(1)$_\eta$ is trivial.

We focus on the case that $\alpha=\delta+1>\eta$ for some $\delta$.
Let $\beta=\sT\pred(\alpha)$.  By induction, (1)$_\beta$ holds.  Let
$\mu=\crit(E_\delta)$, so $\mu<\rho_n^{\sM_\beta}$. By \ref{rem:*3} either
$\mu \leq\lambda_\beta$ or $\kappa_\beta\leq\mu$.

\begin{case}\label{c:easy} $\kappa_\beta<\mu$.\end{case}
Then (1)$_\alpha$ follows from (1)$_\beta$ and the
$\Sigma_0$-elementarity of $i_{\beta,\alpha}$.

\begin{case}\label{c:muleqk} $\mu \leq \kappa_\beta$.\end{case}

Fix $\gamma\in[\lambda_\alpha,\kappa_\alpha]$. We first establish that
the hull property holds at $\gamma$ for $\sM_\alpha$.  If $\gamma<\mu$
then this is as in Case \ref{c:easy} so assume $\mu\leq\gamma$.  Let
$\gamma^*\in\sM_\beta$ be least such that
$i_{\beta,\alpha}(\gamma^*)\geq\gamma$. So
$\mu\leq\gamma^*\in[\lambda_\beta,\kappa_\beta]$, and by (1)$_\beta$,
$\sM_\beta$ has the hull property at $\gamma^*$.  If $\mu=\gamma$ then
$\mu=\gamma^*$ and $\pow(\mu)^{\sM_\beta}=\pow(\mu)^{\sM_\alpha}$
(see \ref{rem:*0}), and the hull property of $\sM_\beta$ at $\mu$
then implies it of $\sM_\alpha$ at $\mu$. So assume $\mu<\gamma$.

Every $A\in\pow(\gamma)^{\sM_\alpha}$ is of the form
$A=[a,f]^{\sM_\beta}_{E_\delta}$, for some $a\in\nu(E_\delta)^{<\om}$
and $f\colon \mu^{\|a\|}\to\pow(\gamma^*)^{\sM_\beta}$ such that $f$ is
given by a generalized $r\bfSigma_n$ term over $\sM_\beta$.  In fact,
$f \in \sM_\beta$. For
$\gamma^*,\mu\leq\kappa_\beta<\rho_n^{\sM_\beta}$, so $f$ is coded by
a bounded subset of $\rho_n^{\sM_\beta}$ which is generalized
$r\bfSigma_n$ over $\sM_\beta$.  If
$\rho_n^{\sM_\beta}>(\kappa_\beta^+)^{\sM_\beta}$
this is clear, and if 
$\rho_n^{\sM_\beta}=(\kappa_\beta^+)^{\sM_\beta}$, it is because there
is $\gamma<\rho_n^{\sM_\beta}$ such that $\ran(f) \subseteq
\sM_\beta|\gamma$; for if $f$ is unbounded then certainly $n\geq 1$,
and as in \cite[p.66]{fsit} one can then use $f$ to give a generalized
$r\bfSigma_n$ definition of a subset $W$ of $\kappa_\beta$ giving a
wellorder of length $(\kappa_\beta^+)^{\sM_\beta}$, but then
$W\in\sM_\beta$, contradiction.

So in fact $f\in\sM_\beta$ and $A=i_{\beta,\alpha}(f)(a)$.  By the
hull property at $\gamma^*$ there is $f'\in\sM_\beta$ such that
\[ f'\in\Hull_0^{\sM_\beta}(\gamma^*\un i_{\chi,\beta}``\sM_\chi) \]
and such that $f'(b)\inter\gamma^*=f(b)$ for all
$b\in\mu^{<\om}$. Therefore, letting
\[ A'=i_{\beta,\alpha}(f')(a),\] we
have $A'\inter\gamma= A$, and
\[ A'\in\Hull_0^{\sM_\alpha}(a\un i_{\beta,\alpha}``\gamma^*\un
i_{\chi,\alpha}``\sM_\chi).\]
Now $i_{\beta,\alpha}``\gamma^*\sub\gamma$ (by choice of
$\gamma^*$). Let us see that we may assume $a\sub\gamma$, giving
\[ A'\in\Hull_0^{\sM_\alpha}(\gamma\un i_{\chi,\alpha}``\sM_\chi),\]
completing the proof of the hull property at $\gamma$.

Well, $a\sub\nu(E_\delta)<i_{\beta,\alpha}(\mu)$. If $\mu<\gamma^*$
then this gives $a\sub\gamma$. Otherwise $\mu=\gamma^*$. Since
$\gamma^*\in[\lambda_\beta,\kappa_\beta]$ and $\mu$ is afm in
$\sM_\beta$, we therefore have $\mu=\lambda_\beta$ or
$\mu=\kappa_\beta$. If $\mu=\lambda_\beta$ then
$\gamma=\lambda_\alpha$ (since $\gamma\geq\lambda_\alpha$ and
$\gamma\leq i_{\beta,\alpha}(\gamma^*)$), and therefore again
$a\sub\gamma$. So suppose $\lambda_\beta<\mu=\kappa_\beta$. So
$\kappa_\beta$ is a successor afm of $\sM_\beta$. Therefore, $E_\delta$
is an order $0$ measure, and we may take $a=\{\kappa_\beta\}$. But
$\gamma>\mu=\kappa_\beta$, so again $a\sub\gamma$, as required.

Next we examine the generators $G_\alpha$.

\begin{subcase}\label{sc:muleql} $\mu\leq\lambda_\beta$.\end{subcase}

We claim ($\dagger 1$): $G_\alpha=i_{\beta,\alpha}``G_\beta$.

Let us prove ($\dagger 1$). By \ref{rem:*5},
$i_{\beta,\alpha}``G_\beta=G_\alpha\inter\rg(i_{\beta,\alpha})$.

If $\lambda=\kappa$, this implies ($\dagger 1$), since
then for every $\eps\in[\eta,\xi]_\sT$ we have $\lambda_\eps=\kappa_\eps$, and
therefore either $G_\eps=\emptyset$ or $G_\eps=\{\kappa_\eps\}$.

So assume $\lambda<\kappa$. Let $\gamma\in G_\alpha$; so
$\gamma\in[\lambda_\alpha,\kappa_\alpha]$. Let $\gamma^*$ be least
such that $i_{\beta,\alpha}(\gamma^*)\geq\gamma$; so
$\gamma^*\in[\lambda_\beta,\kappa_\beta]$.
If $i_{\beta,\alpha}(\gamma^*)=\gamma$ then \ref{rem:*5} implies
that $\gamma^*\in G_\beta$. So assume $i_{\beta,\alpha}(\gamma^*)>\gamma$;
therefore, $\lambda_\beta<\gamma^*\leq\kappa_\beta$.
As before, there is $f\in\sM_\beta$ and
$a\in\nu(E_\delta)^{<\om}$ such that $f\colon \mu^{\|a\|}\to\gamma^*$ and
$i_{\beta,\alpha}(f)(a)=\gamma$. By the hull property at $\gamma^*$,
there is $g\in\sM_\chi$ and $b\in(\gamma^*)^{<\om}$ such that
\[ f=i_{\chi,\beta}(g)(b)\inter(\gamma^*)^2.\]
But then
\[ i_{\chi,\alpha}(g)(i_{\beta,\alpha}(b))(a)=\gamma,\] and
$(a\un i_{\beta,\alpha}(b))\sub\gamma$, so $\gamma\notin G_\alpha$,
contradiction. This proves ($\dagger 1$).

Finally, suppose $\kappa$ is afm in $\sM_\xi$; we must see that $X$ is a
closed set of inaccessibles of $\sM_\alpha$, where
\[ X=\{\kappa_\alpha\}\un G_\alpha\cut\sup i_{\eta,\alpha}``\kappa_\eta.\]
We have $\kappa_\alpha$ afm, and so
inaccessible, in $\sM_\alpha$. If $\lambda=\kappa$ then $X=\{\kappa_\alpha\}$,
so assume $\lambda<\kappa$. Then
$i_{\beta,\alpha}$ is continuous at each $\gamma\in X'$, where
\[ X'=\{\kappa_\beta\}\un G_\beta\cut\sup i_{\eta,\beta}``\kappa_\eta, \]
since by (1)$_\beta$ every $\gamma\in X'$ is inaccessible in
$\sM_\beta$ and
\[ \mu\leq\lambda_\beta<\gamma<\rho_n^{\sM_\beta}.\]
But $X=i_{\beta,\alpha}``X'$, so $X$ is closed. This completes the proof of
(1)$_\alpha$ in this subcase.

\begin{subcase} $\mu>\lambda_\beta$.\end{subcase}

By the case and subcase hypotheses, $\mu=\kappa_\beta>\lambda_\beta$.

We claim ($\dagger
2$): $G_\alpha=(i_{\beta,\alpha}``G_\beta)\un\{\kappa_\beta\}$.

Let us prove $(\dagger 2$). As before, since $\lambda_\beta<\kappa_\beta$ we
have that $E_\delta$
is type 1. Therefore $\kappa_\beta$ is the only
$i_{\beta,\alpha}$-generator, and ($\dagger 2$) then follows from the hull
property for $i_{\chi,\beta}$ at $\kappa_\beta$, like in the previous
subcase.

By $(\dagger 2)$,
\[ G_\alpha=(G_\beta\inter\kappa_\beta)\un\{\kappa_\beta\}\un Y,\]
where $Y=\{i_{\beta,\alpha}(\kappa_\beta)\}$ if $\kappa_\beta\in G_\beta$, and
$Y=\emptyset$ otherwise. Combined with (1)$_\beta$, this readily gives
(1)$_\alpha$ in this subcase.

This completes this subcase, Case \ref{c:muleqk} and the successor step of the
induction.

We mostly leave the case that $\alpha$ is a limit ordinal to the reader. 
However, let us observe
why
\[ \{\kappa_\alpha\}\un G_\alpha\cut\sup i_{\eta,\alpha}``\kappa_\eta\]is
closed. Fix $\gamma<\kappa_\alpha$, and let
$X=G_\alpha\inter(\gamma+1)$ and $Y=X\cut\sup i_{\eta,\alpha}``\kappa_\eta$. It
suffices to see that $Y$ is closed. 

We claim ($\dagger 3$): for any
limit $\beta\in(\chi,\xi]_\sT$, we have
\[ G_\beta=\bigcup_{\beta'\in[\chi,\beta)_\sT}i_{\beta',\beta}``G_{\beta'}.\] 

Indeed, $(\dagger 3)$ follows readily from \ref{rem:*5}.

Now let $\alpha'\in[\eta,\alpha)_\sT$ and $\gamma'<\kappa_{\alpha'}$ be such
that
$i_{\alpha',\alpha}(\gamma')=\gamma$. Then $X=i_{\alpha',\alpha}``X'$ where
$X'=G_{\alpha'}\inter(\gamma'+1)$, by ($\dagger 3$) and (1)$_\beta$ for
$\beta\in[\alpha',\alpha)_\sT$. Let $Y'=X'\cut\sup
i_{\eta,\alpha'}``\kappa_\eta$. Then $Y=i_{\alpha',\alpha}``Y'$, and 
$Y'$ is a closed set of non-afm inaccessibles of $\sM_{\alpha'}$, and
$Y'\sub\rho_n^{\sM_{\alpha'}}$, so $i_{\alpha',\alpha}$ is continuous at each
point of $Y'$, so $Y$ is closed.

This completes our discussion of the limit case, and so completes our proof
of (1).

We now prove (2). So suppose that $\kappa=\rho_n^{\sM_\xi}$. 
Let (2)$_\xi$ be the conjunction of
(2)(a) and (2)(b), relative to $\xi$.

If $\lambda=\kappa$, then
$\lambda_\alpha=\kappa_\alpha=\rho_n^{\sM_\alpha}$ for all
$\alpha\in[\eta,\xi]_\sT$, and $G_\alpha$ is at most
$\{\kappa_\alpha\}$.\footnote{Although we have
$\rho_n^{\sM_\alpha}\in\Hull_k^{\sM_\alpha}(\rg(i_{\chi,\alpha})\un\rho_n^{
\sM_\alpha})$ when $k=n$, the same might not hold when $k=0$. So
$G_\alpha\neq\emptyset$ is possible.} By \ref{rem:*5} we therefore have
$i_{\eta,\alpha}``G_\eta=G_\alpha$. The rest is trivial in this case.

Suppose $\lambda<\kappa$. Again \ref{rem:*5} gives (2)$_\alpha$(a). We have
$\lambda\in\rg(i_{\chi,\xi})$ by Claim \ref{c2}. For each
$\alpha\in[\chi,\xi)_\sT$, we have $\crit(i_{\alpha,\xi})<\rho_n^{\sM_\alpha}$
since there are no drops in degree in $(\chi,\xi]_\sT$. Now an induction like in
the proof of (1), but simpler, shows that (2)$_\alpha$(b) holds for
all $\alpha\in[\chi,\xi]_\sT$.\end{proof}

\section{Main proofs}

\begin{proof}[Proof of Lemma \ref{strongerlem}]
For simplicity, we only directly prove the conclusion of \ref{strongerlem}
with the assumption of ``$\ad+{V=\lr}$''
replaced by ``$\sM_\omega^\#$ exists and is
iterable in $V^{\col(\om,\pow(\RR))}$''.\footnote{This implies
$\ad^{\lr}$, but is stronger; see \cite{hodlr} for the proof of this and the
analysis of $\hod^{\lr}$ under this assumption.} This argument combined with
the argument of \cite[Section 7]{hodcore} then shows that the conclusion actually
follows from ``$\ad+V=\lr$''. Regarding the argument of \cite[\S7]{hodcore}
(and with notation as there), we only need the analysis of
$(\hod|\Theta)^{\J_\gamma(\RR)}$. (Thus, we do not need the arguments of
\cite{hodcore} analysing $\hod$ above $\Theta$.)

We recall a few facts from the analysis of $\hod^{\lr}|\Theta$ using $\mws$ (see
\cite{hodlr} or \cite{hodcore}). There is a directed system $\sF$ defined in
$\lr$,
consisting of premice and iteration maps, whose direct limit is
$\hod^{\lr}|\Theta$. For this we need consider only iteration trees
which are finite stacks $(\sT_0,\ldots,\sT_{n-1})$ of normal trees $\sT_i$, such
that for each $i+1<n$, the main branch of $\sT_i$ does not drop; call such
trees \textit{relevant}. There is a unique
strategy $\Sigma$ for $\mws$ having domain the set of
relevant trees on $\mws$, and which is an $(\om,\om,\om_1+1)$-strategy on its
domain. Let $G$ be generic for $\Coll(\om,\pow(\RR))$ and let $\Sigma'$ be the
corresponding strategy for $V[G]$. Then $\Sigma\sub\Sigma'$; this follows from
the homogeneity of the forcing.

Now in $V[G]$ there is a stack
$\left<\sT_i\right>_{i<\om}$ of
normal iteration trees such that:
\begin{itemize}
\item $\sT_0$ is on $\sQ_0=\sM_\om^\#$.
\item For each $i<\om$, $\sT_i\in\HC^V$.
\item Each $\sT_i$ has a non-dropping main branch and the first model of $\sT_{i+1}$ is the last 
model of $\sT_i$.
\item The stack is via $\Sigma'$ (equivalently, $\left<\sT_i\right>_{i<n}$ is
via $\Sigma$ for each $n<\om$).
\item Let $\sQ_\om$ be the direct limit of the $\sQ_i$'s under the iteration
maps. Let $\delta_0^\sN$ denote the least Woodin cardinal of a model 
$\sN$. Then
$\Theta=\delta_0^{\sQ_\om}$ and
$V_\Theta^{\hod^{\lr}}=V_\Theta^{\sQ_\om}$.
\item Let $i<\om$. Let $j_i:\sQ_i\to\sQ_\om$ be the $\Sigma'$-iteration map. Let
$\gamma<\delta_0^{\sQ_i}$. Then $j_i\rest(\sQ_i|\gamma)\in\lr$.
\end{itemize}

Now let $\kappa<\Theta$ be a cardinal of $\lr$ such that $\om_1<\kappa$. Let $A$ be the
set of measurables of $\hod^{\lr}$ below $\kappa$. Suppose $A$ has ordertype
$<\kappa$. Then the set of afm's  of $\sQ_\om$ below $\kappa$ also has ordertype
$<\kappa$ (for $\sQ_\om$, every afm is finely measurable since $\sQ_\om$ is not
type 2).\footnote{In fact, ``measurable'' implies ``finely measurable'' for
$\sQ_\om|\delta_0^{\sQ_\om}$, by \cite[Section 4]{mim}.}
Work in $\sQ_\om$. Note then that the afm limits of afm's $\leq\kappa$
are bounded by some $\theta<\kappa$. Let $X$ be the set
of afm's in the interval $(\theta,\kappa)$. If $X$ is bounded above by
some $\gamma<\kappa$, let $\mu=1$ and $\kappa_0=\kappa$. Otherwise let
$\left<\kappa_\alpha\right>_{\alpha<\mu}$ enumerate $X$, in strictly
increasing order. By choice of $\theta$, this sequence is
discontinuous everywhere. In either case, $\mu<\kappa$, and
in fact by increasing $\theta$ if
need be, we will assume $\mu<\theta<\kappa_0$. 

Now in $V[G]$, fix $n<\om$ such that $\theta,\kappa\in\rg(j_n)$. Let $j_n(\kappabar)=\kappa$.

For $\alpha<\mu$ let $\gamma_\alpha$ be the sup of all afm's $\gamma$
of $\sQ_\om$ such that $\gamma<\kappa_\alpha$. So
$\gamma_0\leq\theta$ and for $\alpha>0$,
$\gamma_\alpha=\sup_{\beta<\alpha}\kappa_\beta$. We have
$\gamma_\alpha<\kappa_\alpha$ by choice of $\theta$. Let $G_\alpha$ be
the set of all $j_n$-generators in the interval
$[\gamma_\alpha,\kappa_\alpha)$. Since
\[ \kappa=\theta\un\bigcup_{\alpha<\mu}[\gamma_\alpha,\kappa_\alpha)\]
    then $\theta\un\bigcup_{\alpha<\mu}G_\alpha$ includes all
    generators for $j_n$ below $\kappa$. Therefore,
\[
\sQ_\om|\kappa\sub\Hull_0^{\sQ_\om}\left(\theta\un\bigcup_{\alpha<\mu}
G_\alpha\un j_n``\sQ_n\right). \]
However, if $\gamma<\kappa$ is not a $j_n$-generator, by \ref{rem:*5}(b) there
is $f:\kappabar^{<\om}\to\kappabar$ and $a\in\gamma^{<\om}$ such that $f\in\sQ_n$ and
$\gamma=j_n(f)(a)$. Therefore,
\[
\sQ_\om|\kappa\sub\Hull_0^{\sQ_\om|(\kappa^+)^{\sQ_\om}}
\left(\theta\un\bigcup_{\alpha<\mu}G_\alpha\un j_n``(\sQ_n|(\kappabar^+)^{\sQ_n})\right).
\]
But $j_n\rest(\kappabar^+)^{\sQ_n}\in\lr$, and this segment of $j_n$
suffices to compute $\left<G_\alpha\right>_{\alpha<\mu}$. Therefore,
$\lr$ sees the previous fact.  Moreover, we claim that each $G_\alpha$
has ordertype $\leq\om_1$ (in fact, exactly $\om_1$). Therefore, the
previous fact gives a surjection
$(\sQ_n\cross\theta\cross\mu\cross\om_1)^{<\om}\onto\kappa$ in $\lr$. Since
$\sQ_n$
is countable and $\theta,\mu,\om_1<\kappa$, this shows that $\kappa$ is not a
cardinal in $\lr$, a contradiction.

So fix $\alpha<\mu$. Fix $m\geq n$ such that $\alpha\in\rg(j_m)$ and
let $i\geq m$. Let $G^i_\alpha$ be the set of $j_{n,i}$-generators in
the interval $[\gamma^i_\alpha,\kappa^i_\alpha)$, where ``superscript
  $i$'' denotes preimage under $j_i$. We would like to apply Lemma
  \ref{techlem} to deduce that $j_i``G^i_\alpha=G_\alpha\inter\sup
  j_i``\kappa^i_\alpha$.  Given this, then $G_\alpha\inter\sup
  j_i``\kappa^i_\alpha$ has ordertype $<\om_1$ (since
  $\sQ_i\in\HC^V$).  But $G_\alpha=\bigcup_{i\geq
    m}j_i``G^i_\alpha$, by \ref{rem:*5}. So $G_\alpha$ has ordertype
  $\leq\om_1$. (In fact $G_\alpha$ has ordertype $\om_1$, since enough
  normal iterates are absorbed by $\left<\sT_i\right>_{i<\om}$.)

So we just need to see that Lemma \ref{techlem} applies to the
iteration $\left<\sT_i\right>_{i<\om}$, with $\sM_\chi=\sQ_n$,
$\sM_\eta=\sQ_m$, and the ordinal $\kappa^m_\alpha$. We must see that
$\sQ_m$ has the hull property, relative to $j_{n,m}$, at every point
in $[\gamma^m_\alpha,\kappa^m_\alpha]$. Trivially, $\sQ_n$ has the
hull property, relative to $\id$, at every point in
$I=[\gamma_0^n,\kappa^n]$. As in the proof of Lemma \ref{techlem}, an
induction along on the branch $b$ leading from $\sQ_n$ to $\sQ_m$
shows that for every $\beta\in b$, $\sM_\beta$ has the hull property,
relative to $i_{\sQ_n,\beta}$, at every point in
$i_{\sQ_n,\beta}(I)$. This works because $\sQ_n$ has no afm limits of
afm's in $I$. Therefore, $\sQ_m$ has the hull property where required. This completes 
the proof.

(Our use of Lemma~\ref{techlem} can be reduced to the restricted
version described in Remark~\ref{rem:normalstack}. For this version can be
applied inductively to each $\sT_i$ in turn, and \ref{rem:*5} can be quoted
when passing to the direct limit of the stack.)\end{proof}

We now proceed to some $\ZFC$ results which we will apply inside $\HOD^{L(\RR)}$
in our proof of \ref{mainthm} and \ref{simultaneous}. These are variants of the
well-known fact that under $\ZFC$, if $\kappa$ is either a measurable or a
limit of measurables, then $\kappa$ is J\'{o}nsson.

Given $n<\om$ and measures $\mu_i$ over $X_i$ for $i<n$, we write
$\prod_{i<n}\mu_i$, or
\[ \mu_0\cross\ldots\cross\mu_{n-1},\] for the standard
product measure $\mu$ over $\prod_{i<n}X_i$. That
is, $A\in\mu$ iff
\[ \text{for }\mu_0\text{-almost all }x_0\text{, }\ldots\text{, for
}\mu_{n-1}\text{-almost
all }x_{n-1}\text{, }(x_0,\ldots,x_{n-1})\in A.\] If each $x_i$ is a finite set of
ordinals, we might blur the distinction between $(x_0,\ldots,x_{n-1})$ and
$x_0\un\ldots\un x_{n-1}$.

Part (a) of the following lemma is due to Prikry; see \cite{Pr} 
and \cite[8.7]{Kan}.
The remaining parts are straightforward variants and have proofs similar to that
result. We include the proof of all parts here for completeness.

\begin{lem}\label{secondlem} Assume $\zfc$.
Let $\kappa$ be either measurable or a limit of measurables. Then:
\begin{enumerate}
\item[\textup{(}a\textup{)}]
$[\kappa]^{<\om}_{<\kappa}\rightarrow[\kappa]^{<\om}_{\cof(\kappa)}$.

\item[\textup{(}b\textup{)}] If $\cof(\kappa)$ is measurable then
$[\kappa]^{<\om}_{<\cof(\kappa)}\rightarrow[\kappa]^{<\om}_{\om}$.

\item[\textup{(}c\textup{)}] Suppose either $\cof(\kappa)=\om$ or
$\cof(\kappa)$ is measurable. Let $\lambda\in[\om_1,\kappa]$ be a cardinal. Let
\[ F:[\kappa]^{<\om}\to\lambda. \]
Then there is $A\sub\kappa$ such that $\|A\|=\kappa$ and
\[ \|\lambda\cut F``[A]^{<\om}\|=\lambda,\]
and in fact, $\lambda\cut F``[A]^{<\om}$ contains a size $\lambda$ club
subset of $\lambda$.
\end{enumerate}
\end{lem}

\begin{rem}
Assume $\kappa$ is a limit of measurables. The proof that $\kappa$ is
J\'onsson (of which the proof of \ref{secondlem}(c) is a variant) can easily be
extended to show that if $F:[\kappa]^{<\om}\to\kappa$ then there is
$A\sub\kappa$ such that $\card(A)=\kappa$ and
\[ \card(\kappa\cut F``[A]^{<\om})=\kappa.\]
However, if $\cof(\kappa)>\om$ and $\cof(\kappa)$ is not J\'{o}nsson, then
it is easy to see that there is a function $F:[\kappa]^{<\om}\to\kappa$ and a
club subset $C$ of $\kappa$ of size $\cof(\kappa)$, such that for every
$A\sub\kappa$ that is cofinal in $\kappa$, we have
$C\sub F``[A]^{<\om}$, and
therefore, $F``[A]^{<\om}$ is stationary.
\end{rem}

\begin{proof}
If $\kappa$ is regular, then (a) is trivial, (b) just asserts that if $\kappa$
measurable then $\kappa$ is Rowbottom, and (c) follows from the arguments for
the nonmeasurable case. So we assume that $\kappa$ is a singular limit of
measurables.

Let $\mu=\cof(\kappa)<\kappa$. Fix a strictly increasing sequence
$\left<\kappa_\alpha\right>_{\alpha<\mu}$ of measurables $<\kappa$,
whose supremum is $\kappa$, with $\mu<\kappa_0$ and $\lambda<\kappa_0$ if
$\lambda<\kappa$, and such that
for each $\alpha<\mu$, $\gamma_\alpha<\kappa_\alpha$ where
$\gamma_\alpha=\sup_{\beta<\alpha}\kappa_\beta$. Fix a normal measure
$U_\alpha$ on each $\kappa_\alpha$, and if $\mu$ is measurable, fix a normal
measure $U$ on $\mu$.

First we prove (a) and (b); initially we work on both together. Fix
$\lambda<\kappa$ and $F:[\kappa]^{<\om}\to\lambda$.

For $n<\om$ let $T_n={^n}(\om\cut\{0\})$. For each
\[ a=\{a_0<\ldots<a_{\|a\|-1}\}\in{[\mu]}^{<\om} \]
such that $\|a\|\geq 1$ and each
$t\in T_{\|a\|}$, fix a sequence $\left<X_{a,t,i}\right>_{i<\|a\|}$ such that
each
$X_{a,t,i}\in U_{a_i}$ and
$X_{a,t,i}\inter\gamma_{a_i}=\emptyset$, and $F$ is
constant over $X_{a,t}$, where
\[ X_{a,t}=\prod_{i<\|a\|}[X_{a,t,i}]^{t(i)}. \]

For each $\alpha<\mu$, let $X_\alpha=\bigcap I$ where $I$ is the set
of all $X=X_{a,t,i}$ such that
$X\sub[\gamma_\alpha,\kappa_\alpha)$. There are at most $\mu$-many such
  $X$, so $X_\alpha\in U_\alpha$.

We now prove (a). Let (for the proof of (a))
\[ A=\bigcup_{\alpha\in\mu}X_\alpha.\]
We claim that
\[ \|F``[A]^{<\om}\|\leq\mu, \]
as required. For if $b\in {[A]}^{<\om}$
then there
is a unique pair $(a,t)$ such that $b\in X_{a,t}$, but $F$ is constant
over $X_{a,t}$, and there are only $\mu$-many such pairs $(a,t)$. This
completes the proof of (a).

We now prove (b), so assume $\lambda<\cof(\kappa)$. Let
\[ G\colon [\mu]^{<\om}\to\bigcup_n \left({^{(T_n)}}\lambda\right),\]
where for any $a\in[\mu]^{<\om}$,
\[ G(a)\colon T_{\|a\|}\to\lambda \]
is such
that $G(a)(t)=F(b)$ for some (every) $b\in X_{a,t}$. Since $\lambda<\mu$,
$\lambda^\om<\mu$, so we can fix $X\in U$ such that for each $n<\om$,
$G$ is constant over $[X]^n$.

Let (for (b))
\[ A=\bigcup_{\alpha\in X}X_\alpha.\]
We claim that
\[ \|F``[A]^{<\om}\|\leq\om,\]
as required. For if $b\in [A]^{<\om}$,
the value of $F(b)$ depends only on the ``type'' $t$ of $b$. That is,
$F(b_1)=F(b_2)$ whenever there are $a_1,a_2\in [X]^{<\om}$ such that
$\|a_1\|=\|a_2\|$, and $t\in
T_{\|a_1\|}$ such that $b_1\in X_{a_1,t}$
and $b_2\in X_{a_2,t}$. But there are only $\om$-many such pairs $(\|a\|,t)$.
This completes the proof of (b).

We now prove (c). We are given $\lambda,F$.

\setcounter{case}{0}
\begin{case}\label{case:lambda<mu} $\lambda<\mu$.\end{case}
So $\mu$ is measurable. Let $A\sub\kappa$ witness (b). Then $A$ works. If
$\cof(\lambda)>\om$ then this is immediate. Suppose $\cof(\lambda)=\om$. Let
$\left<\gamma_n\right>_{n<\om}$ be an increasing sequence of uncountable regular
cardinals, cofinal in $\lambda$. Let
\[ C_n=\gamma_n\cut((\sup F``[A]^{<\om})+1).\]
Let $C=\bigcup_{n<\om}C_n$. Then $C$ is club in $\lambda$, and is as required.

\begin{case}\label{case:om<mu<lambda<kappa&cof(lambda)/=mu_or_cof(lambda)=om}
$\mu<\lambda<\kappa$ and either $\cof(\lambda)\neq\mu$ or
$\cof(\lambda)=\om$.\end{case}
If $\cof(\lambda)=\om$ then use (a) combined with the argument for
(c), Case \ref{case:lambda<mu}.

If $\cof(\lambda)>\mu$ then the result follows from (a).

So suppose $\om<\cof(\lambda)<\mu$. Let the sets $X_{a,t}$ and $X_\alpha$ be
defined as in the proof of (a).  For each $n<\om$ and $t\in
 T_n$, let
 \[ F_t:[\mu]^n\to\lambda\] be defined by
$F_t(a)=F(u)$ where $u\in X_{a,t}$. Let $Y_t\in U$ be such that
$F_t``[Y_t]^n$ is bounded in $\lambda$. Let $Y=\bigcap_tY_t$. Let
$A=\bigcup_{\alpha\in Y}X_\alpha$. Then $F``[A]^{<\om}$ is bounded in
$\lambda$, so $A$ suffices.

\begin{case}\label{case:mu=lambda<kappa} $\om<\mu=\lambda<\kappa$.\end{case}
So $\mu$ is measurable. Let $X_{a,t}$ and $X_\alpha$ be defined as before, and
let $F_t$ be defined as in Case
\ref{case:om<mu<lambda<kappa&cof(lambda)/=mu_or_cof(lambda)=om}.
 
 Let $X_t\in U^n$ be such that for all $a,c\in X_t$ and $i<n$,
 we have $F_t(a)<a_i$ iff $F_t(c)<c_i$. In fact, take $X_t$ such that if
 $F_t(a)<a_i$ for $a\in X_t$, then $F_t(a)=F_t(c)$ whenever $a,c\in
 X_t$ are such that $a\rest i=c\rest i$.
 
 We can fix $X\in U$,
 and a sequence of functions $\left<G^m_i\right>_{m,i<\om}$, where
 \[G^m_i:[\mu]^m\to\mu,\] and for all $a\in [X]^m$, $G^m_i(a)\geq\max(a)$,
 and such that for each $n,t$ with $t\in T_n$, there are $m,i$, with
 $m\leq n$, such that for all $a\in [X]^n$,
 \[ F_t(a)=G^m_i(a\rest m).\]
(This includes constant functions $G^0_i$.)
 
 Now define a club $C\sub\mu$, with strictly increasing enumeration
 $\left<\delta_\alpha\right>_{\alpha<\mu}$, as follows. Let
 $\delta_0<\mu$ be least such that $\delta_0\notin\rg(G^0_i)$ for all
 $i<\om$. Given $\delta_\alpha$, let $\delta_{\alpha+1}$ be the least
 $\delta$ which is closed under all functions $G^m_i$ and
 $X\inter(\delta_\alpha,\delta)\neq\emptyset$. This determines
 $C$.
 
 Let $B=X\cut C$. Note that $B$ has ordertype $\mu$ and
 \[ G^m_i``[B]^{<\om}\inter C=\emptyset.\]
 Let
 \[ A=\bigcup_{\alpha\in B}X_\alpha. \]
Then $A$ has ordertype $\kappa$ and
\[ F``[A]^{<\om}\inter C=\emptyset, \]
so $A$ suffices.

\begin{case} $\om<\mu<\lambda\leq\kappa$ and $\cof(\lambda)=\mu$.\end{case}
Fix $\left<\lambda_\alpha\right>_{\alpha<\mu}$, a strictly increasing,
continuous sequence $\sub\lambda$, such that for each
$\alpha$, $\lambda_{\alpha+1}$ is a cardinal. Let $W:\lambda\to\mu$ be defined
by
$W(\beta)=\alpha$ where $\beta\in[\lambda_\alpha,\lambda_{\alpha+1})$. Let
  $G=W\com F$. By Case \ref{case:mu=lambda<kappa}, there is $A\sub\kappa$ of
  ordertype $\kappa$ and such that $G``[A]^{<\om}$ is nonstationary in
  $\mu$. Then $\lambda\cut F``[A]^{<\om}$ contains a club in $\lambda$ of size
$\lambda$.

\begin{case} $\om=\mu<\lambda=\kappa$.\end{case}
We will argue similarly to Case \ref{case:mu=lambda<kappa}. Let $T$ be the set
of functions $t\in{^{<\om}}\om$ such
that if $n=\lh(t)\neq 0$ then $t(n-1)\neq 0$. Given $t\in T$, define the
measure
\[ U_t=\prod_{i<\|t\|}U_i^{t(i)}.\]
For each $t\in T$, fix
$Y_t\in U_t$ such that there is $m<\om$ such that
$F``Y_t\sub\kappa_m$. For each
$i<\om$, fix $Y_i\in U_i$, with $Y_i\sub[\kappa_{i-1},\kappa_i)$
  (where $\kappa_{-1}=0$), and such that for each $t$,
\[ \prod_{i<\|t\|}[Y_i]^{t(i)}\sub Y_t. \]

Let $I$ be the set of pairs $(m,s)$ such that $m<\om$ and
$s\in\om^{m+1}$. There are sequences $\left<X_i\right>_{i<\om}$ and
$\left<H_{m,s}\right>_{(m,s)\in I}$ such that, with
\[ X_s=\prod_{i<\lh(s)} [X_i]^{s(i)},\]
we have
\begin{itemize}
\item $X_i\in U_i$ and $X_i\sub Y_i$,
\item $H_{m,s}:X_s\to[\kappa_{m-1},\kappa_m)$, where $\kappa_{-1}=0$,
\item for all $u\in X_s$, $H_{m,s}(u)\geq\max(u)$,
\item for each $t\in T$, there is $(m,s)$ such that
\begin{itemize}
 \item[(i)] either
\begin{itemize}
\item $s=t\rest\lh(s)$ for some $k$, or
\item $s=t\conc\left<0,\ldots,0\right>$, or
\item letting $j=\lh(s)-1$, we have $s\rest j=t\rest j$ and
$s(j)<t(j)$;
\end{itemize}
\item[(ii)] for all $u$ such that
\[ u\in\prod_{i<\|t\|}[X_i]^{t(i)},\]
we have $F(u)=H_{m,s}(u\rest l)$,
where $l=\sum_{i<\lh(s)}s(i)$.
\end{itemize}
\end{itemize}
This can be seen like in Case \ref{case:mu=lambda<kappa}.

Now let $C_n\sub\kappa_n$ be the club of points
$\gamma\in(\kappa_{n-1},\kappa_n)$ such
that for each $s\in\om^{n+1}$ we have
$H_{n,s}``[\gamma]^{<\om}\sub\gamma$. There are clubs $D_n\sub C_n$
and sets $A_n\sub\kappa_n$, each of ordertype $\kappa_n$, such that
$A_n\sub X_n\cut D_n$. Pick closed sets $D'_n\sub D_n$ such that $D'_n$ is
bounded in $\kappa_n$ and $D=\bigcup_{n<\om}D'_n$ has cardinality $\kappa$.

Let $A=\bigcup_{n<\om}A_n$.
Then $A$ is as required, as $D$ is club in $\kappa$,
\[ \|A\|=\kappa=\|D\|\]
and
$F``[A]^{<\om}\inter D=\emptyset$.
\end{proof}

We can now prove the main theorems.
We only explicitly prove \ref{mainthm}; an examination of
its proof also yields \ref{simultaneous}.

\begin{proof}[Proof of Theorems \ref{mainthm},
\ref{simultaneous}.]
  We first prove \ref{mainthm}(b),(d). Work in $L(\RR)$.

 Let $\kappa\in[\om_1,\Theta)$ be a cardinal and
let $F\colon[\kappa]^{<\om}\to\lambda\leq\kappa$.
  Let $\mu=\cof(\kappa)$ and let $f:\mu\to\kappa$ be
  cofinal. Fix $x\in\RR$ such that $f,F\in\hod_x$. We have
$\hod_x\sats\zfc+$``Either $\kappa$ is measurable or is a limit of measurables,
and either $\mu=\om$ or $\mu$ is measurable'', by Corollary \ref{weakercor}
and \cite[8.25]{outline}. So Lemma \ref{secondlem} applies there, yielding a
suitably homogeneous set $A\sub\kappa$. But then $A$ works in $V=\lr$ also.

 Part (b) also gives (a) (if $\cof(\kappa)=\om$ then $\kappa$ is Rowbottom).
Part (d), in the case that $\lambda=\kappa$, gives (c) (that $\kappa$ is
J\'onsson).
  \footnote{Here is a slightly alternative argument for J\'onssonness. Let
  $G\colon [\kappa]^{<\om}\to\kappa$. We need a set $A\sub\kappa$ of ordertype
  $\kappa$ such that $G``[A]^{<\om}\neq\kappa$. If $\kappa$ is regular
  then let $x\in\RR$ be such that $G\in\hod_x$ and use the fact
  that $\kappa$ is measurable, and therefore J\'onsson, in
  $\hod_x$. So assume $\mu=\cof(\kappa)<\kappa$.  If $\mu>\om_1$
  then let $\lambda=\om_1$ (then $\lambda<\mu$); if $\mu\leq\om_1$
  then let $\lambda=\om_2$ (then $\lambda<\kappa$ since $\kappa$ is
  singular). Let $F:[\kappa]^{<\om}\to\lambda+1$ be defined by
  $F(a)=\min(G(a),\lambda)$. By part (b),
  there is $A\sub\kappa$ of ordertype $\kappa$ such that if
  $\lambda=\om_1$ then $\|F``[A]^{<\om}\|\leq\om$, and if $\lambda=\om_2$
  then $\|F``[A]^{<\om}\|\leq\mu\leq\om_1$. In either case, it follows that
  $G``[A]^{<\om}\neq\kappa$, as required.}
\footnote{We didn't actually need the full analysis of $\hod|\Theta$ for the
proofs of either \ref{mainthm} or \ref{simultaneous}. For let $\psi$ be the
assertion that one of them fails. By
the coding lemma, $\psi$ is $\Sigma^2_1$. Then,
letting $\gamma$ be least such that
$\J_\gamma(\RR)\sats$``$\psi+\ZF^-+\pow(\pow(\RR))$ exists'',
it suffices to analyse $(\hod|\delta^2_1)^{\J_\gamma(\RR)}$, using the
argument of \cite{hod_core_below_theta}, combined with a reflection argument
like that in \cite[Section 7]{hodcore}.}
\end{proof}

\section{Acknowledgements}
The first author was supported in part by NSF grant DMS-1201290.
The third author was supported in part by NSF grant DMS-0801189, amendment 
no. 002; thanks are due to the staff at BSU for their help dealing with the grant paperwork.
\bibliographystyle{plain}
\bibliography{biblio}

\end{document}